\documentclass[11pt,a4paper,twoside]{article}
\usepackage{amsthm,amsfonts,amsmath,amscd,amssymb}
\usepackage{latexsym}
\usepackage{euscript}
\usepackage{enumitem}
\usepackage{graphicx}
\usepackage{tikz}
\usepackage{caption}
\usepackage{subcaption}
\usepackage{cite}
\usepackage[utf8]{inputenc}
\usepackage[english]{babel}
\usepackage{xcolor}
\usepackage[hyperpageref]{backref}
\newcommand{\aaa}{{\bf a}} 
\newcommand{\bbb}{{\bf b}} 
\newcommand{\ccc}{{\bf c}} 
\newcommand{\ddd}{{\bf d}} 
\newcommand{\eee}{{\bf e}} 

\usepackage{jmpag} 

\newcommand{\Int}{\text {int} \ }

\newcommand{\R}{\mathbb R} 
\newcommand{\Z}{\mathbb Z} 
\newcommand{\C}{\mathbb C} 

\begin{document}


\author{Dmitry V. Bolotov}
\address{B. Verkin ILTPE of NASU,  47 Nauky Ave., Kharkiv, 61103, Ukraine}
\email{bolotov@ilt.kharkov.ua}
\title[On 2-convex surfaces]{On 2-convex  non-orientable surfaces    in four-dimensional Euclidean space}
\maketitle

\newcommand{\ep}{\varepsilon}
\newcommand{\eps}[1]{{#1}_{\varepsilon}}
\begin{abstract}
	We prove that a 2-convex closed surface $S\subset E^4$  in the four-dimensional Euclidean space $E^4$, which is either  $C^2$-smooth or polyhedral, provided that each vertex is incident to at most five edges, admits a mapping of degree one to a two-dimensional torus, where the degree is assumed to be$\mod 2$ if $S$ is nonorientable. As a corollary, we show that the projective plane and the Klein bottle do not admit such a 2-convex embedding in $E^4$.
\end{abstract}
\section {Introduction}

Let us recall the following definition (see \cite[\S 16]{Ze}).

\begin{definition}\label{def}
	A subset $C$ of  the Euclidean space $E^n$ is called $k$-convex if through each point $x\in E^n\setminus C$ there passes a $k$-dimensional plane $\pi_x$ that does not intersect $C$.
\end{definition} 
Note that for the case $k=n-1$  we obtain one of the equivalent definitions of usual  convexity. If we, moreover,  replace  $E^n$ by the complex affine space  $\C^n$ and  take the hyperplane $\pi_x$    to be complex, we obtain the definition of linear convexity  of  $C\subset \C^n$ (see \cite[Definition 1.6]{Ze}).   Recall, that the concept of linear convexity in two-dimensional complex  plane $\C^2$ was first introduced in \cite{BP}  and has gained  great importance  in complex analysis in several variables (see \cite{Ma},\cite{A}). Topological properties of linearly convex sets and their real analogues are represented in \cite{Ze}.

Clearly, linearly convex sets in $\C^n$ are $(n-2)$ - convex in $\R^{2n}\simeq \R^n\oplus i \R^n=\C^n$.    It is not hard to see that the Clifford torus  ($S^1\times S^1\subset \C^2$) is a linearly convex subset of the complex affine plane  $\C^2$ (see \cite[Example 1.3]{Ze}), in particular,   this gives  an  example of  a 2- convex subset  of $E^4$.  

  Yu. Zelinsky  asked (see \cite[Question 30.5a]{Ze}):  {\it Is there a 2-convex embedding of a compact $K$, which is a homological  sphere $S^2$, in the Euclidean space $E^4$?} 

In papers \cite{B1}, \cite{B2} we gave a partial answer to this question and showed that there is no 2-convex embedding $f:S^2\to E^4$ which is  $C^2$ or  $PL$ - embedding such that the  valence of  vertices does not exceed five.   For simplicity  denote  such $PL$ embeddings  by $PL(5)$. 

In the present paper  we show that our result obtained for the sphere is actually a corollary of  the following   topological  property of   2-convex closed surfaces, which  are  either $C^2$-smooth or $PL(5)$  embedded in $E^4$. Namely, we prove that  such a surface  admits a mapping of degree one onto a two-dimensional torus (Theorem \ref{deg 1}).   As a corollary, we show that  the projective plane and  Klein bottle  do not admit $C^2$- smooth or $PL(5)$  embedding in $E^4$ (Corollary \ref{main}). Note that by the Gauss-Bonnet theorem, the Klein bottle does not admit $PL(5)$-triangulations.

\section{Preliminaries}
\subsection{Degree mod 2 of a mapping}

Let $f: M\to N$ be a differentiable mapping  between  two $n$ - dimensional differentiable closed manifolds. 
\begin{definition}[see \cite {M}]\label{def1}
	Let $y\in N$ 	be a regular value, then we define the degree $\mod 2$  of 
	$f$ (or $\deg_2f$)	by  
	\begin{equation}\label{eq11}
		deg_2 f:=\#  f^{-1}y \ (\mod 2)
		\end{equation}

	It can be shown that the degree $\mod 2$ does not depend on the regular value 	$y$.

As in the orientable   case,  using  Poincar\'e duality with coefficients $\Z_2$,  we have
\begin{equation}\label{eq12}
		deg_2 f:= f^*{\bf u},
\end{equation} 
where $\bf u$ is the  generator of $H^n(N;\Z_2)\simeq \Z_2$ and $f^*: H^n(N;\Z_2)\to H^n(M;\Z_2)$ is the homomorphism of the cohomology groups induced by $f$.  
	\end{definition}

Let us recall the cohomology ring structure of the  torus $T^2$, the projective plane $P^2$ and the Klein bottle $K^2$. 

\renewcommand{\arraystretch}{1.8} 
\renewcommand{\tabcolsep}{0.25cm}   
\begin{table}[h]\footnotesize 
	\caption{Ring structure} \label{1}
	\begin{tabular}{|c|c|c|c|}
		\hline
		A surface S & $T^2$ & $K^2$ & $P^2$\\
		\hline
		Cohomology group {$H^1(S;\Z_2)$}: & $\Z_2\oplus \Z_2$ & $\Z_2\oplus \Z_2$ & $\Z_2$ \\ 
		generators& \aaa,\bbb & \ccc,\ddd & \eee \\
		\hline
		Cohomology ring $H^*(S;\Z_2)$ : & $\Z_2[\aaa,\bbb]/( \aaa^2,\bbb^2)$&  $\Z_2[\ccc,\ddd]/ (\ccc^2,\ddd^3, \ddd^2-\ccc\ddd)$& $\Z_2[\eee]/\eee^3$\\
		\hline
	\end{tabular}

\end{table}
	Here $\Z_2[x_1,\dots,x_n]$  denotes the polynomial ring over  $\Z_2$ in $n$ variables. 

\subsection {Construction}
Recall the construction that underlies the proof of the main result in the case of the sphere (see \cite{B1}, \cite{B2}). Denote by $S$ a closed surface which is $C^2$-smooth or $PL(5)$  embedded into the Euclidean space $E^4$.  Since $S$ is compact, there exists a ball $B^4 \subset E^4$ that contains $S$. We will decrease the radius of $B^4$, until its boundary $S^3 := \partial B^4$ touches  $S$ at some point $p$.   Let  $T_pS^3$ denotes  the 3-dimensional plane  in $E^4$ tangent to $S^3$ at the point $p$.
As was shown in  \cite{B1}, \cite{B2} there exists  
 another 3-dimensional plane $\Pi$ parallel  and sufficiently close to  $T_pS^3$, which divides $E^4$ into  closed  half-spaces $E^4_+\supset p$ and $E^4_-$ such that  the intersection $\gamma:= \Pi\cap S= S_+\cap S_-$, where   $S_{\pm} = E^4_{\pm}\cap S$,  satisfies the  properties represented in Table \ref{2} below.

\renewcommand{\arraystretch}{1.8} 
\renewcommand{\tabcolsep}{0.1cm}   
\begin{table}[h]\footnotesize
	\caption{Properties of $\gamma$.} \label{2}
	\begin{tabular}{|c|c|}
					\hline
		$C^2$ - case & PL(5)-case \\ 
				\hline
		$\gamma $ is a connected $C^2$-smooth curve,   &  $\gamma $ is  a closed $k$-segmented polygonal chain, \\
			$\gamma \subset \partial L_{\gamma}$,  	where $L_{\gamma}$ denotes the  convex hull of  $\gamma$ &  $3\leq k\leq 5$
								\\ 
		\hline
				$p\in S_+$  and $S_+$ is homeomorphic to the disk $D^2$ & $S_+ =C\gamma$ is the cone with vertex $p$ and base $\gamma$.\\ 
				\hline 
	\end{tabular}

\end{table}

These properties are obvious in  $PL(5)$ - case.  In $C^2$ - case we can represent the surface $S$ in a small neighborhood of the point $p$ as follows:
 \[
 \begin{cases}
	x^3 = f(x^1,x^2) \\
	x^4 = g(x^1,x^2), 
\end{cases}
\]
where  $\{x_i , i = 1,\dots,4\}$  are  the Euclidean coordinates in $E^4$ such that the point $p$  is the origin and the coordinate frame $\{e_i , i = 1, \dots , 4\}$ has the property that $\{e_1,e_2\}$ defines the basis of the tangent plane $T_p S$  and $e_3$  is orthogonal to the tangent plane $T_pS^3$, and $f$ is a convex function.   Thus,  the desired 3-dimensional plane $\Pi$ can be taken as $\{x^3=\varepsilon\}$, where $\varepsilon$ is chosen such that the curve $f(x^1,x^2)=\varepsilon$ defines a convex closed curve in the 3-dimensional  plane $\Pi_0:=\{x^4=0\}$.   

	Observe that $l_x:=\pi_x \cap\Pi$ is a  line if $x\in L_{\gamma}\setminus \gamma$ (see Definition \ref{def}).   
The basic observation we made for the case of sphere  in \cite{B1}, \cite{B2}   is  transferred one-to-one to the case of an arbitrary closed closed surface at the form of the following theorem.

\begin{theorem} \label{th1} Let  $S\subset E^4$ be a 2-convex closed surface, which  is  $C^2$ or $PL(5)$ -  embedded in $E^4$.  Let $\Pi$ be a 3-dimensional plane in $E^4$ satisfying the conditions above  (see Table \ref{2}). Then one of the following possibilities occurs:
	\begin{enumerate}
		\item  There exist $x\in L_{\gamma}\setminus \gamma$ and $\pi_{x}$ from Definition \ref{def}, such that $[\gamma]$ is a  generator of $\pi_1(\Pi\setminus l_x)\cong \Z$\footnote{In fact, such a situation is impossible (see the proof of Theorem \ref{deg 1}) .}.
		\item   There exist $x_1, x_2 \in \Int L_{\gamma}$ and $\pi_{x_1}, \pi_{x_2}$ from Definition \ref{def},  such that: 
		\begin{enumerate}
		\item [(a)]  $\pi_{x_1}$ and $ \pi_{x_2}$ are in a general position.  
			\item [(b)] $l_{x_1}\cap l_{x_2}=\emptyset$;
			\item  [(c)]$[\gamma] =[a][b][a]^{-1}[b]^{-1}$, where $a,b$ are the  circles of the bouquet $S^1\vee S^1$ representing  the generators of the group $\pi_1(\Pi\setminus (l_{x_1}\cup l _{x_2}))\cong \Z * \Z$ (see Fig. \ref{r1}).
			
			\end{enumerate} 
		\end{enumerate}

	\end{theorem}
	\begin{figure} [h]
	{\includegraphics[scale=0.3]{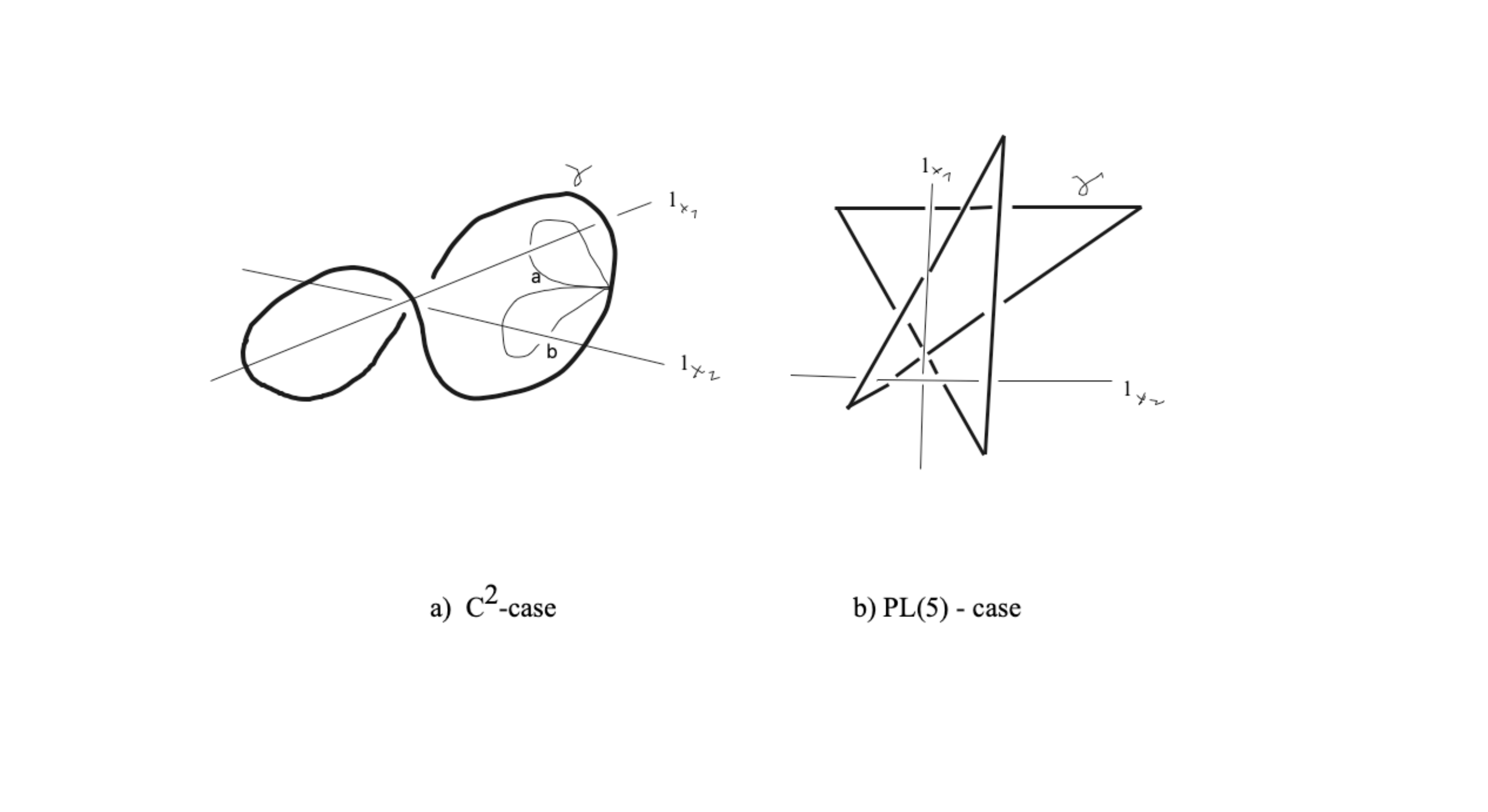}}
	\caption{ Link}
	\label{r1}
\end{figure}
{\it Sketch of proof for $C^2$-case}.  If $\gamma$ is flat then item $1$  holds obviously.  Suppose,  $\gamma$ is not flat,  then  $\gamma$  separates $\partial L_{\gamma}$ into two connected parts  $A_i, \ i=1,2,$ homeomorphic to open discs.   If there exist a point $x\in  \Int L_{\gamma}$ and $\pi_{x}$  such that $l_x\cap A_i, \  i\in \{1,2\}$  is one point, then item $1$ holds obviously.  Otherwise,  $\Int L_{\gamma} =\bigcup_i C_i,, \ i=1,2,$ where $C_i$ is characterized as follows: 
\begin{quote}
$ x_i\in C_i \Leftrightarrow $  There exists a plane $ \pi_{x_i}$ through $x_i$ from Definition \ref{def} such that $l_{x_i}\cap A_i\not=\emptyset$
 \footnote {In this case $l_{x_i}\cap A_i$ consists of two points.}.  
 \end{quote}
 
Observe that  $C_i \not= \emptyset$. Indeed,  let $y_i\in  A_i$ be a smooth boundary point, i.e. there exists only one supporting plane $T_{y_i}$ through $y_i$. Such points are everywhere dense in $A_j$ \cite {L}.  If $l_{y_i}\subset  T_{y_i}$ we can move the plane $\pi_{y_i}$ a little bit to get $l_{y_i}\cap \Int L_{\gamma}\not = \emptyset$. On the other hand, obviously, $C_i$ are open subsets in $\Int L_{\gamma}$.   From the connectivity of $\Int L_{\gamma}$ it follows that $C_1\cap C_2\not =\emptyset$.  This means that  there exist $x\in \Int L_{\gamma}$ and  planes $\pi_{i}, \ i=1,2,$ through $x$ from Definition \ref{def}  such that $l_{1}\cap l_{2}= x$ and $l_{i}\cap A_i\not=\emptyset$, where $l_i:=\pi_i\cap \Pi$.  Now, by a small perturbation of the planes $\pi_{i}$ to the planes $\pi_{x_i}$  through some points $x_1,x_2 \in \Int L_{\gamma}$ we can get items $2_{(a)}$ and $2_{(b)}$ of Theorem \ref{th1}  with the mutual arrangement of the straight lines $l_{x_1}, l_{x_2}$ and the curve $\gamma$ as a  nontrivial link as shown in Fig. \ref{r1} a).  The diagram of this link can be obtained  by the orthogonal projection $p: \Pi\to  \Pi\cap\Pi_0$. 
Item $2_{(c)}$ is checked  directly. 

\section{ Main result}

\begin{theorem}\label{deg 1}
If $ S\subset    E^4$  is a 2-convex closed surface which is either $C^2$  or $PL(5)$ - embedded into $E^4$,  then it admits a   mapping   of degree one  to a two-dimensional  torus.  If $S$ is non-orientable, we assume that  the degree is taken modulo 2.
\end{theorem}

\begin{proof}
    Suppose  $f: S \to E^4$ is   a  2-convex embedding  satisfying the condition of the theorem.  Then Theorem \ref{th1} is satisfied. If item $1$ of Theorem  \ref {th1} holds, then we immediately come  to a contradiction, since $\gamma$ bounds $S_+\simeq D^2$, which contradicts  to $[\gamma] \not=0$ in $\pi_1(\Pi\setminus l_x)$ and therefore $[\gamma] \not=0$ in $\pi_1(E^4\setminus \pi_x)$ since $\Pi\setminus l_x$ is the deformation retract of $E^4\setminus \pi_x$. 
 
 Let us suppose that item $2$ of Theorem  \ref {th1}  is satisfied. Then from item $2_{(a)}$ it follows that  $\pi_{x_1}\cap  \pi_{x_2}$ is a point, which we denote by $O$.    There are two possible cases.
 
 {\it  Case 1. }  $O\in E_-^4$.  In this case we have the deformation retraction $$r_+: E_+^4\setminus (\pi_{x_1}\cup  \pi_{x_2})\to \Pi\setminus (l_{x_1}\cup l_{x_2}),$$ 
 which is defined as follows: 
 $$r_+(x) = l_{Ox}\cap \Pi,$$
 where $l_{Ox}$ denotes the straight line passing through $O$ and $x\in E^4_+$.  
 
 We immediately come  to a contradiction, since $\gamma$ bounds the disk $$S_+\subset E_+^4\setminus (\pi_{x_1}\cup \pi_{x_2})$$ and $[\gamma]$  must be equal to zero in $$\pi_1(E_+^4\setminus (\pi_{x_1}\cup \pi_{x_2}))\simeq \pi_1(\Pi \setminus (l_{x_1}\cup l_{x_2}))\footnote{Recall that  deformation retraction induces an isomorphism of fundamental groups.},$$ which contradicts  to item  $2_{(c)}$  of Theorem \ref{th1}. 
 
 {\it Case 2.}  $O\in E_+^4$.    In this case,  as  above, we have the deformation retraction $$r_-: E_-^4\setminus (\pi_{x_1}\cup  \pi_{x_2})\to \Pi\setminus (l_{x_1}\cup l_{x_2})$$ 
 and  the homotopically  inverse to $r_-$ embedding
 $$ i_- :  \Pi\setminus (l_{x_1}\cup l_{x_2}) \to E_-^4\setminus (\pi_{x_1}\cup  \pi_{x_2}).$$

 Moreover, we have the following homotopy equivalences:  
 
 	\begin{enumerate}
\item [] $ h_1: \Pi\setminus (l_{x_1}\cup l_{x_2}) \overset {r_1}\rightarrow S^1 \vee  S^1\overset {i_1} \hookrightarrow \Pi\setminus (l_{x_1}\cup l_{x_2})$;
 \item []$ h_2: S_-  \overset {r_2}\rightarrow  \vee_p  S_p^1 \overset {i_2}\hookrightarrow S_ -$ (see  Remark \ref{gon}),
\end{enumerate}
where $ \vee_{p=1}^n  S_p^1 $ is a bouquet of circles   generating $\pi_1(S_-)$,   and $r_k$ ,$i_k$  ($k=1,2$ ) are  the deformation retractions and the   homotopically inverse embeddings to them.  

\begin{remark} \label{gon}  Note that the choice of  the bouquet $ \vee_{p=1}^n  S_p^1 $ generating   $\pi_1(S_-)$   and the deformation retraction $r_2: S_-  \to \vee_{p=1}^n  S_p^1 $ are  ambiguous.   Moreover,   we can always choose  circles $S^1_p$ such that $\cap_{p=1}^nS^1_p\cap \gamma $ is a single point (the basepoint   of $S_-$), which we denote by $o$. 
		Fixing the orientations of the circles $S^1_p$ and cutting along them, we obtain a representation of $S_-$ by a 2n-gon  $P$ with a hole bounded by $\gamma$ and oriented sides that are pairwise identified (see Fig.\ref{r3} in the case where $S$ is either $P^2$ or $K^2$).    The deformation retraction $$r_2: P\setminus \Int  S_+ \to \partial P$$  is  shown in the Figure \ref{r3} by the thin arrows. This defines  the equality $[\gamma] = [w] \in \pi_1(S_-)$,   where $w = r_2(\gamma)$.  
	
		\end{remark}

 	\begin{figure} [h]
 	{\includegraphics[scale=0.5]{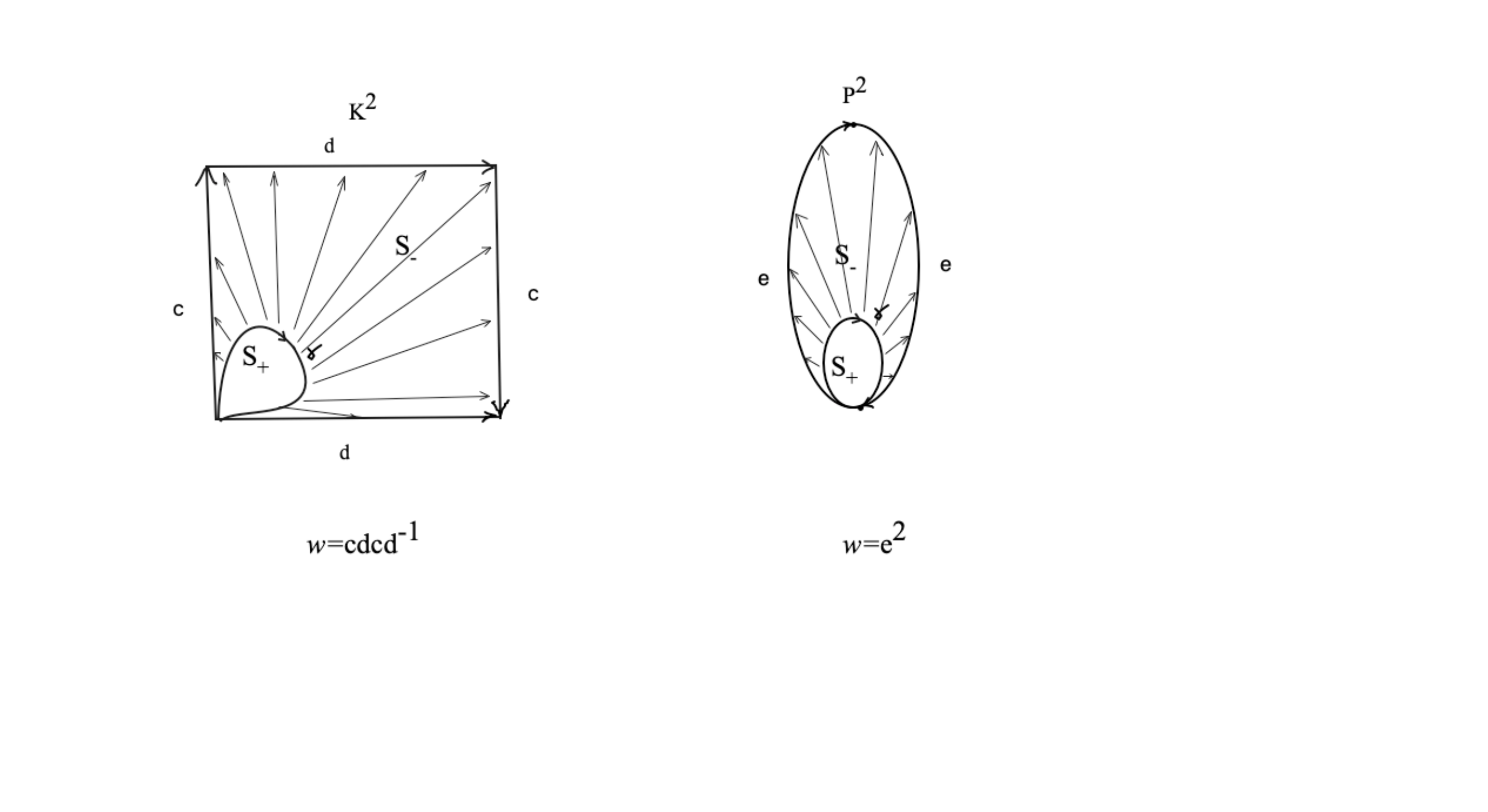}}
 	\caption{ Deformation retraction}
 	\label{r3}
 \end{figure}

 \begin{remark}
 	The marking of the sides of polygons in the  Fig.\ref{r3} are not accidental, since the sides represent homological classes of the group $H_1(S;\Z_2)$ of the  surface $S$   obtained by the gluing of the corresponding sides of the polygons,  which are    Poincar\'e dual to the classes of cohomologies in  $H^1(S;\Z_2)$ denoted by the same letters in Table \ref{1}. Note that the modulo 2  indices of intersection of homological classes  correspond to the multiplication of cohomology classes with coefficients in $\Z_2$.
 \end{remark}
 
 Let us consider the following composition of the mappings: 
 
  $$\psi:=r_1\circ r_-\circ \bar f\circ i_2:   \vee_p  S_p^1\to S^1 \vee  S^1,$$
    where $\bar f$ is defined from the composition: 
    $$  f|_{S_-}:S_-\overset{\bar f} \to E_-^4\setminus (\pi_{x_1}\cup  \pi_{x_2})\hookrightarrow E^4.$$ 
   
Taking into account the definition of $w$ and  item $2_{(c)}$ of Theorem \ref{th1} we have $${\psi}_*[w] = {\psi}_*{r_2}_*[\gamma] =( r_1\circ r_-)_*\bar f_*{i_2}_*{r_2}_*[\gamma]= ( r_1\circ r_-)_*\bar f_*[\gamma]= [a][b][a]^{-1}[b]^{-1},$$  where $*$    indicates the  induced homomorphisms of  fundamental groups.
  
  Let us complete   the natural cellular decomposition of $S^1\vee S^1$  to the cellular decomposition of  the torus $T^2$ as shown in Fig. \ref{r6}.  Observe, that $S$  has a natural  cellular decomposition   generated by $o$, $ \vee_pS^1_p\setminus o$,  $\gamma\setminus o$, $\Int S_-$, $\Int S_+$.  We claim that one can extend  $\psi$  to the continuous  map $\Psi : S\to T^2$ of degree one. 
  
  	\begin{figure} [h]
  		  	{\includegraphics[scale=0.3]{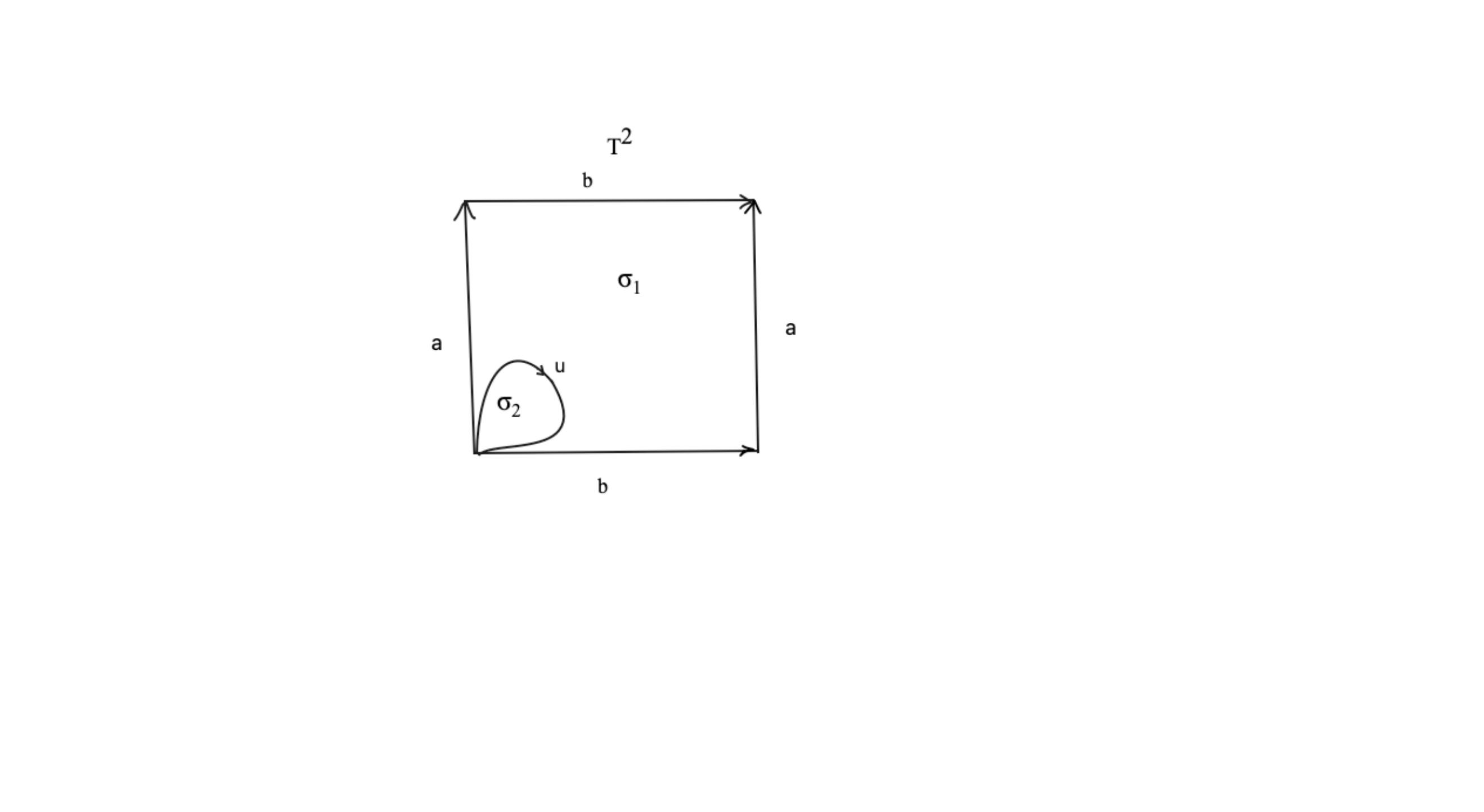}}
  	\caption{ Cellular decomposition of $T^2$ }
  	\label{r6}
  \end{figure}

  One can extend $\psi$  to   the map of 1-skeletons  $\Psi^{(1)}:S^{(1)} \to {T^2}^{(1)}, $  by   means of some  diffeomorphism $\Psi^{(1)}|_{\gamma}: \gamma\to u$ preserving orientations.  By the construction we have  $$(\Psi^{(1)})_*[\gamma]^{-1}[w]=[u]^{-1}[a][b][a]^{-1}[b]^{-1}.$$ 
Thus,  we can extend $\Psi^{(1)}$ to  $S_-$, whose images are   in $\bar\sigma_1$ (see Fig.\ref{r6}) and then  to  $S_+$ by   some   diffeomorphism    $\Psi|_{S_+} : S_+\to  \bar\sigma_2$.    Smoothing $\Psi$ and  leaving them unchanged in the  small disk $B \subset \Int S_+$, we  see that the  degree of $\Psi$  is equal to one (see \eqref{eq11}).  

\end{proof}

\begin{theorem}\label{thm2}
	Neither  the  projective plane   nor  Klein bottle   admit a continuous mapping  onto the torus $T^2$ of degree one $\mod 2$. 
	\end{theorem}

\proof  
Let us suppose that $g: S\to T^2$  is a mapping of a closed surface $S$  to the torus $T^2$ and  $\deg_2 g =1$. Then, using   the  ring structure of   $H^*(T^2;\Z_2)$  (see Table \ref{1}) and   \eqref{eq12} we have: 
  \begin{equation} \label {eq1} 
  	g^*(\aaa\cup \bbb) =g^*\aaa\cup g^*\bbb  \not=0.
  	\end{equation}
 Note,   that  
      \begin{equation} \label {eq2} 
      	\aaa\cup \aaa = \bbb\cup \bbb=0.
      		\end{equation}

{\it Case 1}.  $S=P^2$. In this case, it follows from the ring structure  of $H^*(P^2;\Z_2)$  (see Table \ref{1}) and \eqref{eq1}  that  $g^*\aaa = g^*\bbb =\eee$. But this leads to a  contradiction:
$$  0\not =\eee^2 =g^*\aaa\cup g^*\aaa = g^*(\aaa\cup \aaa)\overset{\eqref{eq2}}=0. $$

{\it Case 2}. $S=K^2$.   Let ${\bf c}, {\bf d}$ be  generators of the group $H^1(K^2;\Z_2)\simeq \Z_2\oplus \Z_2$  (see Table \ref{1}).  If $g^*\aaa = g^*\bbb =\ccc$, then we obtain the contradiction: 
$$  0=\ccc^2 =g^*\aaa\cup g^*\bbb\overset{\eqref{eq1}}{\not=}0. $$ 
Thus,   one of two things happens: 
\begin{itemize}
\item [a)]$g^*\aaa\in\{\ddd, \ddd+\ccc\}$;
 \item [b)] $g^*\bbb\in\{\ddd, \ddd+\ccc\}$.  
\end{itemize}
Without loss of generality we can assume that    $a)$ is satisfied.
 Using the  ring structure of $H^*(K^2;\Z_2)$, 
we have: 
\begin{itemize}
	\item $\ddd^2\not=0$;
	\item   $ (\ccc+\ddd)\cup (\ccc+\ddd) = \ccc^2 +2\ccc\cup \ddd + \ddd^2 = \ddd^2\not=0.$
\end{itemize}
But this yields to a contradiction:
$ 0\not=g^*\aaa \cup g^*\aaa =g^*(\aaa\cup \aaa) \overset{\eqref{eq2}}=0.$

\endproof

 Theorems \ref{deg 1}  and  \ref{thm2} impliy: 

\begin{corollary}\label{main}  
	Neither the  projective plane nor Klein bottle admit  a $C^2$-smooth or  $PL(5)$  embedding in  four-dimensional Euclidean space as a 2-convex surface.   
\end{corollary}

\section {Final remarks}

Finally, we note that the question under consideration for  non-orientable surfaces $M_{\mu}^2$,  $\mu\geq 3$, where $\mu$ denotes the number of M\"obius bands glued into the sphere $S^2$, remains open. Indeed, any such surface is homeomorphic to the torus $T^2$ with $\mu-2$ M\"obius bands glued into it.  The mapping $g: M^2_{\mu}\to T^2$, contracting each such  Mobius band to a point, has $\deg_2g=1$ (see Fig. \ref{r4}) and   Theorem  \ref{thm2}  is wrong for $S= M_{\mu}^2$,  $\mu\geq 3$.

	\begin{figure} [h]
	{\includegraphics[scale=0.2]{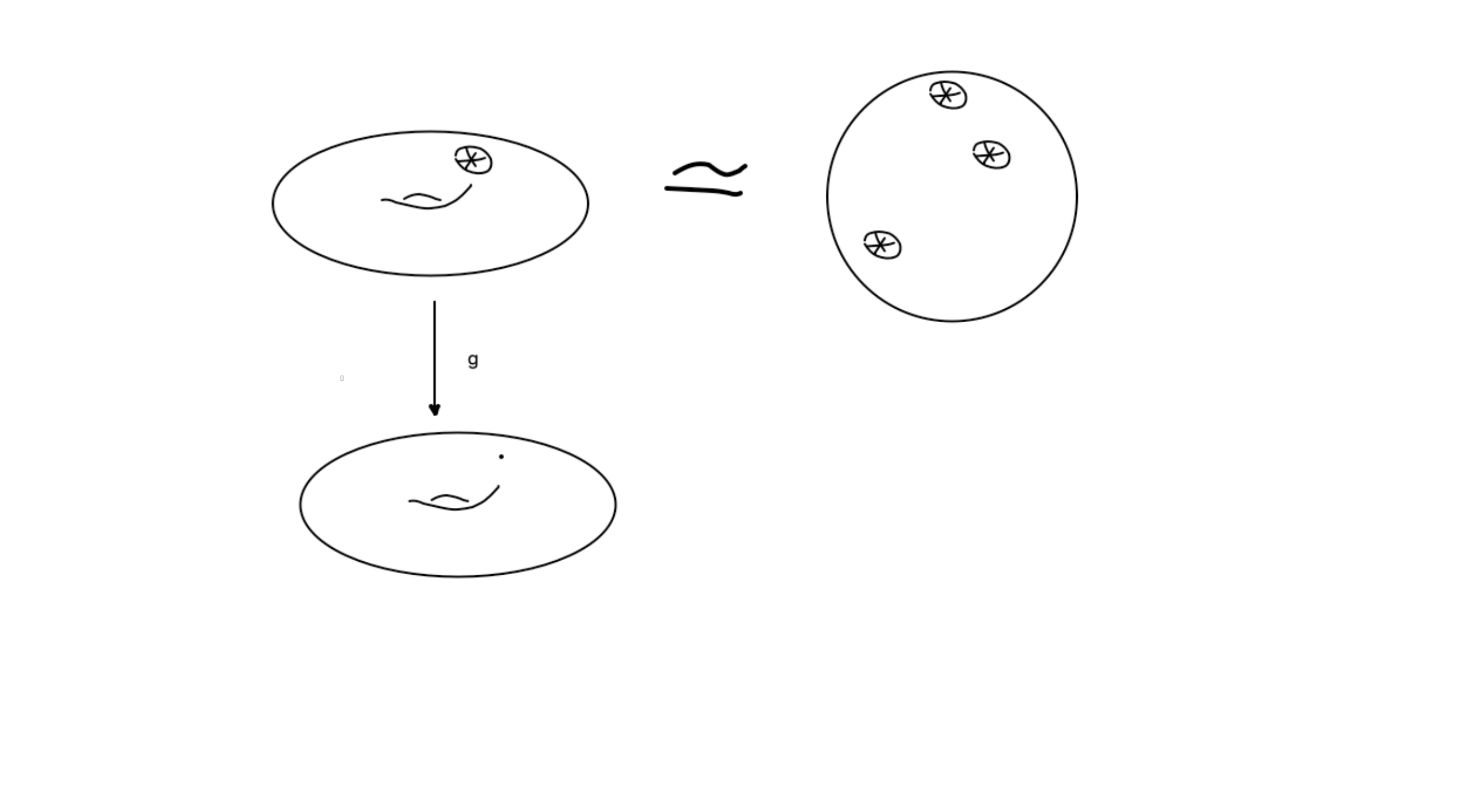}}
	\caption{  Degree one$\mod 2$ map  of $M_{3}^2$ to the torus $T^2$.}
	\label{r4}
\end{figure}


\end{document}